\documentclass{amsart}
\usepackage{amsmath}
\usepackage{amscd}
\usepackage{amssymb}
\usepackage{latexsym,nicefrac}
\input xy
\xyoption{all}

\newcommand{\cO}{\mathcal O}

\renewcommand{\ker}{{\rm ker}\,}
\newcommand{\coker}{{\rm coker}\,}

\newcommand{\Si}{{\rm Sink}}

\DeclareMathOperator{\Cone}{Cone}

\newcommand{\coli}{\operatornamewithlimits{colim}}

\newcommand{\tor}{\mathrm{Tor}}

\newcommand{\comment}[1]{}  

\def\N{\mathbb{N}}
\def\H{\mathbb{H}}

\def\Z{\mathbb{Z}}

\newcommand\triqui{\vartriangleleft}

\newcommand{\weq}{\overset{\sim}{\to}}

\numberwithin{equation}{section} \theoremstyle{plain}
\newtheorem{thm}[equation]{Theorem}
\newtheorem{cor}[equation]{Corollary}
\newtheorem{lem}[equation]{Lemma}
\newtheorem{prop}[equation]{Proposition}

\theoremstyle{definition}

\theoremstyle{remark}
\newtheorem{rem}[equation]{Remark}
\newtheorem{exa}[equation]{Example}
\newtheorem{que}[equation]{Question}

\newtheorem*{ack}{Acknowledgement}
\begin{document}
\bibliographystyle{plain}

\author{ Pere Ara}

\author{Guillermo Corti\~nas}

\title{Tensor products of Leavitt path algebras}
\thanks{The first author was partially supported by
DGI MICIIN-FEDER MTM2008-06201-C02-01, and by the Comissionat per
Universitats i Recerca de la Generalitat de Catalunya. The second
named author was supported by CONICET and partially supported by
grants PIP 112-200801-00900, UBACyTs X051 and 20020100100386, and MTM2007-64074.}
\address{P. Ara\\
Departament de Matem\`atiques\\ Universitat Aut\`onoma de
Barcelona\\ 08193 Bellaterra (Barcelona), Spain}
\email{para@mat.uab.cat}
\address{G. Corti\~nas\\ Dep. Matem\'atica and Instituto Santal\'o\\ Ciudad Universitaria Pab 1\\
1428 Buenos Aires, Argentina}
\email{gcorti@dm.uba.ar}\urladdr{http://mate.dm.uba.ar/\~{}gcorti}

\date{\today}

\begin{abstract}
We compute the Hochschild homology of Leavitt path algebras over a
field $k$. As an application, we show that $L_2$ and $L_2\otimes
L_2$ have different Hochschild homologies, and so they are not
Morita equivalent; in particular they are not isomorphic. Similarly,
$L_\infty$ and $L_\infty\otimes L_\infty$ are distinguished by their
Hochschild homologies and so they are not Morita equivalent either.
By contrast, we show that $K$-theory cannot distinguish these
algebras; we have $K_*(L_2)=K_*(L_2\otimes L_2)=0$ and
$K_*(L_\infty)=K_*(L_\infty\otimes L_\infty)=K_*(k)$.
\end{abstract}

\maketitle

\section{Introduction}
Elliott's theorem \cite{Rord} stating that $\mathcal O _2\otimes
\mathcal O_2\cong \mathcal O_2$ plays an important role in the proof
of the celebrated classification theorem of Kirchberg algebras in
the UCT class, due to Kirchberg \cite{Kirch} and Phillips
\cite{Phill}. Recall that a Kirchberg algebra is a purely infinite,
simple, nuclear and separable C*-algebra. The Kirchberg-Phillips
theorem states that this class of simple C*-algebras is completely
classified by its topological $K$-theory. The analogous question
whether the algebras $L_2$ and $L_2\otimes L_2$ are isomorphic has
remained open for some time. Here $L_2$ is the Leavitt algebra of
type $(1,2)$ over a field $k$ (see \cite{lea}), that is,  the
$k$-algebra with generators $x_1,x_2,x_1^*,x_2^*$ and relations
given by $x_i^*x_j=\delta_{i,j}$ and $\sum _{i=1}^2 x_ix_i^*=1$.

\medskip

In this paper we obtain a negative answer to this question. Indeed,
we analyze a much larger class of algebras, namely the tensor
products of Leavitt path algebras of finite quivers, in terms of
their Hochschild homology, and we prove that, for $1\le n<m\le
\infty$, the tensor products $E=\bigotimes _{i=1}^nL(E_i)$  and
$F=\bigotimes _{j=1}^mL(F_j)$ of Leavitt path algebras of
non-acyclic finite quivers $E_i$, $F_j$, are distinguished by their
Hochschild homologies (Theorem \ref{thm:tensofin}). Because
Hochschild homology is Morita invariant, we conclude that $E$ and
$F$ are not Morita equivalent for $n<m$. Since $L_2$ is the Leavitt
path algebra of the graph with one vertex and two arrows, we obtain
that $L_2\otimes L_2$ and $L_2$ are not Morita equivalent; in
particular they are not isomorphic.

\medskip
Recall that, by a theorem of Kirchberg \cite{kp}, a simple, nuclear
and separable $C^*$-algebra $A$ is purely infinite if and only if
$A\otimes\cO_\infty\cong A$. We also show that the analogue of
Kirchberg's result is not true for Leavitt algebras. We prove in
Proposition \ref{prop:linfi} that if $E$ is a non-acyclic quiver,
then $L_\infty\otimes L(E)$ and $L(E)$ are not Morita equivalent,
and also that $L_\infty\otimes L_\infty$ and $L_\infty$ are not
Morita equivalent.

\medskip

Using the results in \cite{abc} we prove that the algebras $L_2$ and
$L_2\otimes L(F)$, for $F$ an arbitrary finite quiver, have trivial
$K$-theory: all algebraic $K$-theory groups $K_i$, $i\in \Z$, vanish
on them (this follows from Lemma \ref{reg-super} and Proposition
\ref{no-distinguish}). We also compute
$K_*(L(F))=K_*(L_\infty\otimes L(F))$ and that
$K_*(L_\infty)=K_*(L_\infty\otimes L_\infty)=K_*(k)$ is the
$K$-theory of the ground field (see Proposition \ref{LinfKabsorbs}
and Corollary \ref{cor:linfy}). This implies in particular that, in
contrast with the analytic situation, no classification result, in
terms solely of $K$-theory, can be expected for a class of central,
simple $k$-algebras, containing all purely infinite simple unital
Leavitt path algebras, and closed under tensor products. It is worth
mentioning that an important step towards a $K$-theoretic
classification of purely infinite simple Leavitt path algebras of
finite quivers has been achieved in \cite{ALPS}.

\medskip

We refer the reader to \cite{AB}, \cite{AMP} and \cite{ra} for the
basics on Leavitt algebras, Leavitt path algebras and graph
C*-algebras,  and to \cite{RordEnc} for a nice survey on the
Kirchberg-Phillips Theorem.

\medskip

\paragraph{\em Notations.} We fix a field $k$; all vector spaces,
tensor products and algebras are over $k$. If $R$ and $S$ are unital
$k$-algebras, then by an $(R,S)$-bimodule we understand a left
module over $R\otimes S^{op}$. By an $R$-bimodule we shall mean an
$(R,R)$ bimodule, that is, a left module over the enveloping algebra
$R^e=R\otimes R^{op}$. Hochschild homology
of $k$-algebras is always taken over $k$; if $M$ is an $R$-bimodule,
we write
\[
HH_n(R,M)=\tor_n^{R^e}(R,M)
\]
for the Hochschild homology of $R$ with coefficients in $M$; we
abbreviate $HH_n(R)=HH_n(R,R)$.

\section{Hochschild homology}\label{sec:huni}

Let $k$ be a field, $R$ a $k$-algebra and $M$ an $R$-bimodule. The
Hochschild homology $HH_*(R,M)$ of $R$ with coefficients in $M$ was
defined in the introduction; it is computed by the {\em Hochschild
complex} $HH(R,M)$ which is given in degree $n$ by
\[
HH(R,M)_n=M\otimes R^{\otimes n}
\]
It is equipped with the Hochschild boundary map $b$ defined by
\[
b(a_0\otimes a_1\otimes\dots\otimes
a_n)=\sum_{i=0}^{n-1}(-1)^i a_0\otimes\dots\otimes
a_ia_{i+1}\otimes\dots\otimes a_n+(-1)^na_na_0\otimes\dots\otimes
a_{n-1}
\]
If $R$ and $M$ happen to be $\Z$-graded, then $HH(R,M)$ splits into
a direct sum of subcomplexes
\[
HH(R,M)=\bigoplus_{m\in\Z}{}_mHH(R,M)
\]
The homogeneous component of degree $m$ of $HH(R,M)_n$ is the linear
subspace of $HH(R,M)_n$ generated by all elementary tensors
$a_0\otimes\dots\otimes a_n$ with $a_i$ homogeneous and
$\sum_i|a_i|=m$. One of the first basic properties of the Hochschild
complex is that it commutes with filtering colimits. Thus we have

\begin{lem}\label{lem:coli}
Let $I$ be a filtered ordered set and let $\{(R_i,M_i):i\in I\}$ be
a directed system of pairs $(R_i,M_i)$ consisting of an algebra
$R_i$ and and  an $R_i$-bimodule $M_i$, with algebra maps $R_i\to
R_j$ and $R_i$-bimodule maps $M_i\to M_j$ for each $i\le j$. Let
$(R,M)=\coli_{i}(R_i,M_i)$. Then $HH_n(R,M)=\coli_i HH_n(R_i,M_i)$
($n\ge 0$).
\end{lem}

Let $R_i$  be a $k$-algebra and $M_i$ an $R_i$- bimodule $(i=1,2)$.
The K\"unneth formula establishes a natural isomorphism
(\cite[9.4.1]{chubu})
\[
HH_n(R_1\otimes R_2, M_1\otimes M_2)\cong
\bigoplus_{p=0}^nHH_p(R_1,M_1)\otimes HH_{n-p}(R_2,M_2)
\]
Another fundamental fact about Hochschild homology which we shall
need is Morita invariance. Let $R$ and $S$ be Morita equivalent
algebras, and let $P\in R\otimes S^{op}-\mod$ and $Q\in S\otimes
R^{op}-\mod$ implement the Morita equivalence. Then (\cite[Thm.
9.5.6]{chubu})
\begin{equation}\label{morita}
HH_n(R,M)=HH_n(S,Q\otimes_RM\otimes_RP)
\end{equation}

\begin{lem}\label{lem:hhtenso}
Let $R_1,\dots,R_n$ and $S_1,\dots,S_m,\dots$ be a finite and an
infinite sequence of algebras, and let $R=\bigotimes_{i=1}^nR_i$,
$S_{\le m}=\bigotimes_{j=1}^{m}S_j$ and $S=\bigotimes_{j=1}^\infty
S_j$. Assume:
\begin{enumerate}
\item $HH_q(R_i)\ne 0\ne HH_q(S_j)$ \quad ($q=0,1$), ($1\le i\le n$), ($1\le j$).
\item $HH_p(R_i)=HH_p(S_j)=0$ for $p\ge 2$, $1\le i\le n$, $1\le
j$.
\item $n\ne m$.
\end{enumerate}
Then no two of $R$, $S_{\le m}$ and $S$ are Morita equivalent.
\end{lem}
\begin{proof}
By the K\"unneth formula, we have
\[
HH_n(R)=\bigotimes_{i=1}^n HH_1(R_i)\ne 0, \qquad HH_p(R)=0\qquad
p>n
\]
By the same argument, $HH_p(S_{\le m})$ is nonzero for $p=m$, and
zero for $p>m$. Hence if $n\ne m$, $R$ and $S_{\le m}$ do not have
the same Hochschild homology and therefore they cannot be Morita
equivalent, by \eqref{morita}. Similarly, by Lemma \ref{lem:coli},
we have
$$HH_n(S)= \bigoplus_{J\subset \N, |J|=n} \big(\bigotimes_{j\in J}
HH_1(S_j)\big)\otimes \big(\bigotimes _{j\notin J} HH_0(S_j)\big)\,
,
$$
so that $HH_n(S)$ is nonzero for all $n\ge 1$, and thus it cannot be
Morita equivalent to either $R$ or $S_{\le m}$.
\end{proof}

\section{Hochschild homology of crossed products}

Let $R$ be a unital algebra and $G$ a group acting  on $R$ by
algebra automorphisms. Form the crossed-product algebra $S=R\rtimes
G$, and consider the Hochschild complex $HH(S)$. For each conjugacy
class $\xi$ of $G$, the graded submodule $HH^{\xi}(S)\subset HH(S)$
generated in degree $n$ by the elementary tensors $a_0\rtimes
g_0\otimes\dots\otimes a_n\rtimes g_n$ with $g_0\cdots g_n\in\xi$ is
a subcomplex, and we have a direct sum decomposition
$HH(S)=\bigoplus_{\xi}HH^{\xi}(S)$. The following theorem of Lorenz
describes the complex $HH^{\xi}(S)$ corresponding to the conjugacy
class $\xi=[g]$  of an element $g\in G$ as
hyperhomology over the centralizer subgroup $Z_g\subset G$.

\begin{thm}\label{thm:lorenz}\cite{lore}.
Let $R$ be a unital $k$-algebra, $G$ a group acting on $R$ by
automorphisms, $g\in G$ and $Z_g\subset G$ the centralizer subgoup.
Let $S=R\rtimes G$ be the crossed product algebra, and $HH^{\langle
g\rangle}(S)\subset HH(S)$ the subcomplex described above. Consider
the $R$-submodule $S_g=R\rtimes g\subset S$. Then there is a
quasi-isomorphism
\[
HH^{[g]}(S)\weq \H(Z_g,HH(R,S_g))
\]
In particular we have a spectral sequence
\[
E^2_{p,q}=H_p(Z_g,HH_q(R,S_g))\Rightarrow HH_{p+q}^{[g]}(S)
\]
\end{thm}

\begin{rem}
Lorenz formulates his result in terms of the spectral sequence
alone, but his proof shows that there is a quasi-isomorphism as
stated above; an explicit formula is given
for example in the proof of \cite[Lemma 7.2]{proper}.

\end{rem}

Let $A$ be a not necessarily unital $k$-algebra, write $\tilde{A}$
for its unitalization. Recall from \cite{wodex} that $A$ is called
\emph{$H$-unital} if the groups $\tor_n^{\tilde{A}}(k,A)$ vanish for
all $n\ge 0$. Wodzicki proved in \cite{wodex} that $A$ is $H$-unital
if and only if for every embedding $A\triqui R$ of $A$ as a
two-sided ideal of a unital ring $R$, the map
\[
HH(A)\to HH(R:A)=\ker(HH(R)\to HH(R/A))
\]
is a quasi-isomorphism.

\begin{lem}\label{lem:hlorenz}
Theorem \ref{thm:lorenz} still holds if the condition that $R$ be
unital is replaced by the condition that it be $H$-unital.
\end{lem}
\begin{proof} Follows from Theorem \ref{thm:lorenz} and the fact,
proved in \cite[Prop. A.6.5]{proper}, that $R\rtimes G$ is
$H$-unital if $R$ is.
\end{proof}

Let $R$ be a unital algebra, and $\phi:R\to pRp$ a corner
isomorphism. As in \cite{skew}, we consider the skew Laurent
polynomial algebra $R[t_+,t_-,\phi]$; this is the $R$-algebra
generated by elements $t_+$ and $t_-$ subject to the following
relations.
\begin{gather*}
t_+a=\phi(a)t_+\\
at_-=t_-\phi(a)\\
t_-t_+=1\\
t_+t_-=p
\end{gather*}
Observe that the algebra $S=R[t_+,t_-,\phi]$ is $\Z$-graded by
$\deg(r)=0$, $\deg(t_\pm)=\pm 1$. The homogeneous component of
degree $n$ is given by
\[
R[t_+,t_-,\phi]_n=\left\{\begin{matrix}t^{-n}_-R& n<0\\
R & n=0\\
Rt^n_+& n>0\\
\end{matrix}\right.
\]

\begin{prop}\label{prop:monlorenz}
Let $R$ be a unital ring, $\phi:R\to pRp$ a corner isomorphism, and
$S=R[t_+,t_-,\phi]$. Consider the weight decomposition
$HH(S)=\bigoplus_{m\in\Z}{}_{m}HH(S)$. There is a
quasi-isomorphism
\begin{equation}\label{map:monlorenz}
{\ \ }_{m}HH(S)\weq \Cone(1-\phi:HH(R,S_m)\to HH(R,S_m))
\end{equation}
\end{prop}
\begin{proof}
If $\phi$ is an automorphism, then $S=R\rtimes_\phi\Z$, the right
hand side of \eqref{map:monlorenz} computes $\H(\Z,HH(R,S_m))$, and
the proposition becomes the particular case $G=\Z$ of Theorem
\ref{thm:lorenz}. In the general case, let $A$ be the colimit of the
inductive system
\[
\xymatrix{R\ar[r]^\phi& R\ar[r]^\phi & R\ar[r]^\phi&\dots}
\]
Note that $\phi$ induces an automorphism $\hat{\phi}:A\to A$. Now
$A$ is $H$-unital, since it is a filtering colimit of unital
algebras, and thus the assertion of the proposition is true for the
pair $(A,\hat{\phi})$, by Lemma \ref{lem:hlorenz}. Hence it suffices
to show that for $B=A\rtimes_{\hat{\phi}}\Z$ the maps $HH(S)\to
HH(B)$ and $\Cone(1-\phi:HH(R,S_m)\to HH(R,S_m))\to
\Cone(1-\phi:HH(A,B_m)\to HH(A,B_m))$ $(m\in\Z)$ are
quasi-isomorphisms. The analogous property for $K$-theory is shown
in the course of the third step of the proof of \cite[Thm.
3.6]{abc}. Since the proof in {\em loc. cit.} uses only that
$K$-theory commutes with filtering colimits and is matrix invariant
on those rings for which it satisfies excision, it applies verbatim
to Hochschild homology. This concludes the proof.
\end{proof}

\section{Hochschild homology of the Leavitt path algebra}

Let $E=(E_0,E_1,r,s)$ be a finite quiver and let $\hat{E}=(E_0,
E_1\sqcup E_1^*, r,s)$ be the double of $E$, which is the quiver
obtained from $E$ by adding an arrow $\alpha^*$ for each arrow
$\alpha \in E_1$, going in the opposite direction. The {\it Leavitt
path algebra}  of $E$ is the algebra $L(E)$ with one generator for
each arrow $\alpha\in \hat{E}_1$ and one generator $p_i$ for each
vertex $i\in E_0$, subject to the following relations
\begin{gather*}
                      p_ip_j=\delta_{i,j}p_i\, , \qquad (i,j\in E_0)\\
                      p_{s(\alpha)}\alpha =\alpha =\alpha p_{r(\alpha )}\, ,\qquad (\alpha\in
                      \hat{E}_1)\\
                      \alpha^*\beta=\delta_{\alpha,\beta}p_{r(\alpha)}\, , \qquad (\alpha,\beta \in E_1)\\
                      p_i=\sum_{\alpha \in E_1, s(\alpha)=i}\alpha \alpha^* \, , \qquad (i\in E_0\setminus \Si (E))
\end{gather*}
The algebra $L=L(E)$ is equipped with a $\Z$-grading. The grading
is determined by $|\alpha|=1$, $|\alpha^*|=-1$, for $\alpha \in
E_1$. Let $L_{0,n}$ be the linear span of all the elements of the
form $\gamma \nu^*$, where $\gamma$ and $\nu$ are paths with
$r(\gamma)=r(\nu)$ and $|\gamma |=|\nu |=n$. By \cite[proof of
Theorem 5.3]{AMP}, we have $L_0=\bigcup _{n=0}^{\infty} L_{0,n}$.
For each $i$ in $E_0$, and each $n\in \Z^+$, let us denote by
$P(n,i)$ the set of paths $\gamma$ in $E$ such that $|\gamma |=n$
and $r(\gamma)=i$.  The algebra $L_{0,0}$ is isomorphic to $\prod
_{i\in E_0}k$. In general the algebra $L_{0,n}$ is isomorphic to
\begin{equation}\label{L0n}
\Big[ \prod _{m=0}^{n-1}\big( \prod _{i\in \Si
(E)}M_{|P(m,i)|}(k)\big)\Big] \times \Big[ \prod _{i\in
E_0}M_{|P(n,i)|}(k) \Big].
\end{equation}

The transition homomorphism $L_{0,n}\to L_{0,n+1}$ is the identity
on the factors $$\prod _{i\in \Si (E)}M_{|P(m,i)|}(k),$$ for $0\le
m\le n-1$, and also on the factor
$$\prod _{i\in \Si
(E)}M_{|P(n,i)|}(k)$$ of the last term of the displayed formula. The
transition homomorphism
$$\prod_{i\in E_0\setminus \Si (E)}M_{|P(n,i)|}(k)\to \prod_ {i\in
E_0}M_{|P(n+1,i)|}(k)$$ is a block diagonal map induced by the
following identification in $L(E)_0$: A matrix unit in a factor
$M_{|P(n,i)|}(k)$, where $i\in E_0\setminus \Si (E)$, is a monomial
of the form $\gamma\nu ^*$, where $\gamma$ and $\nu$ are paths of
length $n$ with $r(\gamma)=r(\nu)=i$. Since $i$ is not a sink, we
can enlarge the paths $\gamma$ and $\nu$ using the edges that $i$
emits, obtaining paths of length $n+1$, and the last relation in the
definition of $L(E)$ gives
\[
\gamma \nu^*=\sum _{\{\alpha\in E_1\mid s(\alpha )=i\}}(\gamma
\alpha)(\nu\alpha)^*.
\]

Assume $E$ has no sources. For each $i\in E_0$, choose an arrow
$\alpha_i$ such that $r(\alpha_i)=i$. Consider the elements
\[
t_+=\sum_{i\in E_0}\alpha_i,\qquad t_-=t_+^*
\]
One checks that $t_-t_+=1$. Thus, since $|t_\pm|=\pm 1$, the
endomorphism
\begin{equation}\label{map:phi}
\phi:L\to L, \qquad \phi(x)=t_+xt_-
\end{equation}
is homogeneous of degree $0$ with respect to the $\Z$-grading. In
particular it restricts to an endomorphism of $L_0$. By \cite[Lemma
2.4]{skew}, we have
\begin{equation}\label{skewle}
L=L_0[t_+,t_-,\phi].
\end{equation}

Consider the matrix $N_E'=[n_{i,j}]\in M_{e_0}\Z$  given by
\[
n_{i,j}=\#\{\alpha\in E_1:s(\alpha)=i,\quad r(\alpha)=j\}
\]
Let $e_0'=|\Si (E)|$. We assume that $E_0$ is ordered so that the
first $e_0'$ elements of $E_0$ correspond to its sinks. Accordingly,
the first $e_0'$ rows of the matrix $N'_E$ are $0$. Let $N_E$ be the
matrix obtained by deleting these $e_0'$ rows. The matrix that
enters the computation of the Hochschild homology of the Leavitt
path algebra is
$$\begin{pmatrix}
0 \\1_{e_0-e_0'}
\end{pmatrix}-N_E^t\colon \Z^{e_0-e_0'}\longrightarrow \Z^{e_0}.$$
By a slight abuse of notation, we will write $1-N_E^t$ for this
matrix. Note that $1-N_E^t\in M_{e_0\times (e_0-e_0')}(\Z)$. Of
course $N_E=N_E'$ in case $E$ has no sinks.

\begin{thm}\label{thm:hhl1}
Let $E$ be a finite quiver without sources, and let $N=N_E$. For
each $i\in E_0\setminus\Si(E)$, and $m\ge 1$, let $V_{i,m}$ be the
vector space generated by all closed paths $c$ of length $m$ with
$s(c)=r(c)=i$. Let $\Z=<\sigma>$ act on
\[
V_m=\bigoplus_{i\in E_0\setminus\Si(E)}V_{i,m}
\]
by rotation of closed paths. We have:
\[
{}_mHH_n(L(E))=\left\{\begin{matrix} \coker(1-\sigma:V_{|m|}\to V_{|m|})& n=0, m\neq 0\\
                                     \coker(1-N^t)& n=m=0\\
                                     \ker(1-\sigma:V_{|m|}\to V_{|m|})& n=1, m\neq 0\\
                                     \ker(1-N^t)& n=1, m=0\\
                                     0 & n\notin\{0,1\}\end{matrix}\right.
\]
\end{thm}
\begin{proof}
Let $L=L(E)$, $P=P(E)\subset L$ the path algebra of $E$ and
$W_m\subset P$ be the subspace generated by all paths of length $m$.
For each fixed $n\ge 1$, and $m\in\Z$, consider the following
$L_{0,n}$-bimodule
\[
L_{m,n}=\left\{\begin{matrix}L_{0,n}W_mL_{0,n} & m>0\\
L_{0,n}W^*_{-m}L_{0,n} & m<0\end{matrix}\right.
\]
Write $L=L(E)$, and let ${}_mL$ be the homogeneous part of degree
$m$; we have
\[
{}_mL=\bigcup_{n\ge 1} L_{m,n}
\]
If $m$ is positive, then there is a basis of $L_{m,n}$ consisting of
the products $\alpha\theta\beta^*$ where each of $\alpha$, $\beta$
and $\theta$ is a path in $E$, $r(\alpha)=s(\theta)$,
$r(\beta)=r(\theta)$, $|\alpha|=|\beta|=n$ and $|\theta|=m$. Hence
the formula
\[
\pi(\alpha\theta\beta^*)=\left\{\begin{matrix} \theta & \text{ if }\alpha=\beta\\
                                                  0  & \text{else} \end{matrix}\right.
\]
defines a surjective linear map $L_{m,n}\to V_m$. One checks that
$\pi$ induces an isomorphism
\[
HH_0(L_{0,n},L_{m,n})\cong V_m \quad (m>0)
\]
Similarly if $m<0$, then
\[
HH_0(L_{0,n},L_{m,n})=V_{|m|}^*\cong V_{-m}.
\]
Next, by \eqref{L0n}, we have
\[
HH_0(L_{0,n})=k[E\setminus\Si(E)]\oplus\bigoplus_{i\in\Si(E)}k^{r(i,n)}
\]
Here
\[
r(i,n)=\max\{r\le n: P(r,i)\neq\emptyset\}
\]
Now note that, because $L_{0,n}$ is a product of matrix algebras, it
is separable, and thus $HH_1(L_{0,n},M)=0$ for any bimodule $M$. As
observed in \eqref{skewle}, for the automorphism \eqref{map:phi}, we
have $L=L_0[t_+,t_-,\phi]$. Hence in view of Proposition
\ref{prop:monlorenz} and Lemma \ref{lem:coli}, it only remains to
identify the maps $HH_0(L_{0,n},L_{m,n})\to
HH_0(L_{0,n+1},L_{m,n+1})$ induced by inclusion and by the
homomorphism $\phi$. One checks that for $m\ne 0$, these are
respectively the cyclic permutation and the identity $V_{|m|}\to
V_{|m|}$. The case $m=0$ is dealt with in the same way as in
\cite[Proof of Theorem 5.10]{abc}.
\end{proof}

\begin{cor}\label{cor:nonzero} Let $E$ be a finite quiver with at least one non-trivial closed path.
\item[i)] $HH_n(L(E))=0$ for $n\notin\{0,1\}$.
\item[ii)] ${}_mHH_*(L(E))\cong {}_{-m}HH_*(L(E))$ ($m\in\Z$).
\item[iii)] There exist $m>0$ such that ${}_mHH_0(L(E))$ and ${}_mHH_1(L(E))$
are both nonzero.
\end{cor}
\begin{proof}
We first reduce to the case where the graph does not have sources.
By the proof of \cite[Theorem 6.3]{abc}, there is a finite complete
subgraph $F$ of $E$ such that $F$ has no sources, $F$ contains all
the non-trivial closed paths of $E$, $\Si (F)=\Si (E)$, and $L(F)$
is a full corner in $L(E)$ with respect to the homogeneous
idempotent $\sum _{v\in F^0} p_v$. It follows that $HH_*(L(E))$ and
$HH_*(L(F))$ are graded-isomorphic. Therefore we can assume that $E$
has no sources.

The first two assertions are already part of Theorem \ref{thm:hhl1}.
For the last assertion, let $\alpha$ be a primitive closed path in
$E$, and let $m=|\alpha|$. Let $\sigma$ be the cyclic permutation;
then $\{\sigma^i\alpha:i=0,\dots,m-1\}$ is a linearly independent
set. Hence $N(\alpha)=\sum_{i=0}^{m-1}\sigma^i\alpha$ is a nonzero
element of $V_m^\sigma={}_mHH_1(L(E))$. Since on the other hand $N$
vanishes on the image of $1-\sigma:V_m\to V_m$, it also follows that
the class of $\alpha$ in ${}_mHH_0(L(E))$ is nonzero.
\end{proof}

\section{Applications}

\begin{thm}\label{thm:tensofin}
Let $E_1,\dots,E_n$ and $F_1,\dots,F_m$ be finite quivers. Assume
that $n\ne m$ and that each of the $E_i$ and the $F_j$ has at least
one non-trivial closed path. Then the algebras
$L(E_1)\otimes\dots\otimes L(E_n)$ and $L(F_1)\otimes\dots\otimes
L(F_m)$ are not Morita equivalent.
\end{thm}
\begin{proof}
Immediate from Lemma \ref{lem:hhtenso} and Corollary
\ref{cor:nonzero}(iii).
\end{proof}

\begin{exa}
It follows from Theorem 5.1 that $L_2$ and $L_2 \otimes_k L_2$ are
not Morita equivalent. There is another way of proving this, due to
Jason Bell and George Bergman \cite{BB}. By Theorem 3.3 of
\cite{BD}, $\mathrm{l.gl.dim}\, L_2 \le 1$. Using a module-theoretic
construction, Bell and Bergman show that $\mathrm{l.gl.dim(L_2
\otimes_k L_2)} \ge 2$, which forces $L_2$ and $L_2 \otimes_k L_2$
to be not Morita equivalent. Bergman then asked Warren Dicks whether
general results were known about global dimensions of tensor
products and was pointed to Proposition 10(2) of \cite{ERZ}, which
is an immediate consequence of Theorem XI.3.1 of \cite{CE}, and says
that if $k$ is a field and $R$ and $S$ are $k$-algebras, then
$\mathrm{l.gl.dim}\, R + \mathrm{w.gl.dim}\, S \le \mathrm{l.gl.dim}
(R \otimes_k S).$ Consequently, if $\mathrm{l.gl.dim}\, R < \infty$
and $\mathrm{w.gl.dim}\, S > 0$, then $\mathrm{l.gl.dim}\, R <
\mathrm{l.gl.dim}(R \otimes_k S)$; in particular, $R$ and $R
\otimes_k S$ are then not Morita equivalent. To see that
$\mathrm{w.gl.dim}\, L_2 > 0$, write $x_1$, $x_2$, $x_1^*$, $x_2^*$
for the usual generators of $L_2$ and use normal-form arguments to
show that $\{a \in L_2 \mid ax_1 = a+1\} = \emptyset$ and $\{b \in
L_2 \mid x_1b = b\} = \{0\}$. Hence, in $L_2$, $x_1-1$ does not have
a left inverse and is not a left zerodivisor (or see [4]) ; thus,
$\mathrm{w.gl.dim}\, L_2 > 0$.

\end{exa}
\comment{
\begin{exa} It follows from Theorem \ref{thm:tensofin} that the algebras
$L_2$ and $L_2\otimes L_2$ are not Morita equivalent. Here is
another way of proving this, due to Jason Bell and George Bergman
\cite{BB}. Since the weak global dimension of a tensor product of
algebras over a field is at least the sum of their global
dimensions, it suffices to show that $L_2$ has weak dimension $1$.
Since $L_2$ has global dimension $1$, it suffices to show it is not
von Neumann regular. To see this write $x_1,x_2,x_1^*,x_2^*$ for the
usual generators of $L_2$ and observe that (e.g. by the results of
\cite{AB1}) the element $1-x_1$ is a nonzero divisor which is not a
unit.
\end{exa}
} We denote by $L_{\infty}$ the unital algebra presented by
generators $x_1,x_1^*,x_2,x_2^*, \dots$ and relations
$x^*_ix_j=\delta _{i,j}1$.

\begin{prop}
\label{prop:linfi}
 Let $E$ be any finite quiver having at least one non-trivial closed path.
Then $L_{\infty}\otimes L(E)$ and  $L(E)$ are not Morita equivalent.
Similarly $L_{\infty}\otimes L_{\infty}$ and $L_{\infty}$ are not
Morita equivalent.
\end{prop}

\begin{proof}
Let $C_n$ be the algebra presented  by generators $x_1,x_1^*,\dots
,x_n,x_n^*$ and relations $x_i^*x_j=\delta _{i,j}1$, for $1\le
i,j\le n$. Then \begin{equation} \label{eq:directlim}
L_{\infty}=\varinjlim C_n\, ,
\end{equation}
and $C_n\cong L(E_n)$, where $E_n$ is the graph having two vertices
$v,w$ and $2n$ arrows $e_1,\dots ,e_n,f_1,\dots, f_n$, with
$s(e_i)=r(e_i)=v= s(f_i)$ and $r(f_i)=w$ for $1\le i\le n$. (The
isomorphism $C_n\to L(E_n)$ is obtained by sending $x_i$ to
$e_i+f_i$ and $x_i^*$ to $e_i^*+f_i^*$.) It follows from Theorem
\ref{thm:hhl1} and (\ref{eq:directlim}) that the formulas in Theorem
\ref{thm:hhl1} for ${}_mHH_n(L_{\infty})$, $m\ne 0$, hold taking as
$V_{i,m}$ the vector space generated by all the words in
$x_1,x_2,\dots $ of length $m$, and that ${}_0HH_0(L_{\infty})=k$
and ${}_0HH_n(L_{\infty})=0$ for $n\ge 1$. As before, Lemma
\ref{lem:hhtenso} gives the result.
\end{proof}

\begin{thm}\label{thm:tensoinf}
Let $E_1,\dots,E_n$ and  $F_1,\dots,F_m,\dots$ be a finite and an
infinite sequence of quivers. Assume that the number of indices $i$
such that $F_i$ has at least one non-trivial closed path is
infinite. Then the algebras $L(E_1)\otimes\dots\otimes L(E_n)$ and
$\bigotimes_{i=1}^\infty L(F_i)$ are not Morita equivalent.
\end{thm}
\begin{proof}
Immediate from Lemma \ref{lem:hhtenso} and Corollary
\ref{cor:nonzero}(iii).
\end{proof}
\begin{exa}
Let $L^{(\infty )}=\bigotimes_{i=1}^\infty L_2$, and let $E$ be any
quiver having at least one non-trivial closed path. Then $L^{(\infty
)}\otimes L(E)$ and $L(E)$ are not Morita equivalent.
\end{exa}

It would be interesting to know the answer to the following
question:

\begin{que}
\label{que:maps} Is there a unital homomorphism $\phi\colon
L_2\otimes L_2 \to L_2$?
\end{que}

Observe that, to build a unital homomorphism $\phi \colon L_2\otimes
L_2\to L_2$, it is enough to exhibit a {\it non-zero} homomorphism
$\psi\colon L_2\otimes L_2\to L_2$, because $eL_2e\cong L_2$ for
every non-zero idempotent $e$ in $L_2$.

\medskip
\section{$K$-theory}

To conclude the paper we note that algebraic $K$-theory cannot distinguish
between $L_2$ and $L_2\otimes L_2$ or between $L_\infty$ and $L_\infty\otimes L_\infty$.
For this we need a lemma, which
might be of independent interest. Recall that a unital ring $R$ is
said to be {\it regular supercoherent} in case all the polynomial
rings $R[t_1,\dots ,t_n]$ are regular coherent in the sense of
\cite{Gersten}.

\begin{lem}
\label{reg-super} Let $E$ be a finite graph. Then $L(E)$ is regular
supercoherent.
\end{lem}

\begin{proof} Let $P(E)$ be the usual path algebra of $E$. It was observed in the proof
of \cite[Lemma 7.4]{AB} that the algebra $P(E)[t]$ is regular
coherent. The same proof gives that all the polynomial algebras
$P(E)[t_1,\dots ,t_n]$ are regular coherent. This shows that $P(E)$
is regular supercoherent. By \cite[Proposition 4.1]{AB}, the
universal localization $P(E)\to L(E)=\Sigma^{-1}P(E)$ is flat on the
left. It follows that $L(E)$ is left regular supercoherent (see
\cite[page 23]{abc}). Since $L(E)\otimes k[t_1,\dots ,t_n]$ admits
an involution, it follows that $L(E)$ is regular supercoherent.
\end{proof}

\begin{prop}
\label{no-distinguish} Let $R$ be regular supercoherent. Then the
algebraic $K$-theories of $L_2$ and of $L_2\otimes R$ are both
trivial.
\end{prop}

\begin{proof}
Let $E$ be the quiver with one vertex and two arrows. Then $L_2\cong
L(E)$, and we have
$$L_2\otimes R=L_{R}(E).$$
Applying \cite[Theorem 7.6]{abc} we obtain that
$K_*(L_R(E))=K_*(L(E))=0$. The result follows.
\end{proof}

We finally obtain a $K$-absorbing result for Leavitt path algebras
of finite graphs, indeed for any regular supercoherent algebra.

\begin{prop}
\label{LinfKabsorbs} Let $R$ be a regular supercoherent algebra.
Then the natural inclusion $R\to R\otimes L_{\infty}$ induces an
isomorphism $K_i(R)\to K_i(R\otimes L_{\infty})$ for all $i\in
\mathbb Z$.
\end{prop}

\begin{proof}
Adopting the notation used in the proof of Proposition
\ref{prop:linfi}, we see that it is enough to show that the natural
map $R\to R\otimes L(E_n)$ induces isomorphisms $K_i(R)\to
K_i(R\otimes L(E_n))$ for all $i\in \Z$ and all $n\ge 1$. Since $R$
is regular supercoherent the $K$-theory of $R\otimes L(E_n)\cong
L_R(E_n)$ can be computed by using \cite[Theorem 7.6]{abc}. By the
explicit form of the quiver $E_n$, we thus obtain that
$$K_i(R\otimes L(E_n))\cong (K_i(R)\oplus K_i(R))/(-n,1-n)K_i(R).$$
The natural map $R\to  L_R(E_n)$ factors as
$$R\to Rv\oplus Rw \to  L_R(E_n)\, . $$
The first map induces the diagonal homomorphism $K_i(R)\to
K_i(R)\oplus K_i(R)$ sending $x$ to $(x,x)$. The second map induces
the natural surjection
$$K_i(R)\oplus K_i(R)\to (K_i(R)\oplus K_i(R))/(-n,1-n)K_i(R).$$
Therefore the natural homomorphism $R\to L_R(E_n)$ induces an
isomorphism $$K_i(R)\overset\sim\longrightarrow K_i(L_R(E_n)).$$ This concludes the
proof.
\end{proof}
\begin{cor}\label{cor:linfy}
The natural maps $k\to L_\infty\to L_\infty\otimes L_\infty$ induce
$K$-theory isomorphisms $K_*(k)=K_*(L_\infty)=K_*(L_\infty\otimes
L_\infty)$.
\end{cor}
\begin{proof}
A first application of Proposition \ref{LinfKabsorbs} gives
$K_*(k)=K_*(L_\infty)$. A second application shows that for $E_n$ as
in the proof above, the inclusion $L(E_n)\to L(E_n)\otimes L_\infty$
induces a $K$-theory isomorphism; passing to the limit, we obtain
the corollary.
\end{proof}

\begin{ack} Part of the research for this article was carried out during a visit of the second named author
to the Centre de Recerca Matem\`atica. He is indebted to this
institution for its hospitality.
\end{ack}

\end{document}